\documentclass{article}

\usepackage[latin2]{inputenc}
\usepackage{t1enc, anysize, enumerate, amsthm, amsmath, amssymb, epsfig}
\usepackage[unicode]{hyperref}
\newtheorem{theorem}{Theorem}

\newtheorem{conjecture}[theorem]{Conjecture}

\begin{document}

\title{An Erd\H{o}s--Ko--Rado type theorem for\\ subgraphs of perfect matchings}
\author{D\'{a}niel T. Nagy\footnote{HUN-REN Alfr\'ed R\'enyi Institute of Mathematics, Budapest, Hungary. dani.t.nagy@gmail.com\\ Research supported by the National Research, Development and Innovation Office - NKFIH under the grants FK 132060 and PD 137779 and by the J\'anos Bolyai Research Scholarship of the Hungarian Academy of Sciences.}}
\maketitle

\begin{abstract}
Let $M_k$ be a $2n$-vertex graph with $n$ pairwise disjoint edges and let $\mathcal{H}^{(p,s)}(n)$ be the family of subsets of $V(M_n)$ that span exactly $p$ edges and $s$ isolated vertices. We prove that for $n\ge 2p+s$ this family has the Erd\H{o}s--Ko--Rado property: the size of the largest intersecting family equals to the number of sets containing a fixed vertex. The bound $n\ge 2p+s$ is the best possible, improving a recent theorem with $n\ge 2p+2s$ by Fuentes and Kamat.
\end{abstract}

\section{Introduction}

We say that a set system (or family) is {\it intersecting}, if any two of its members has a nonempty intersection.  The Erd\H{o}s--Ko--Rado theorem is a fundamental result in the theory of extremal set systems.

\begin{theorem}[Erd\H{o}s, Ko, Rado (1961) \cite{ekr1961}] \label{ekrthm}
Let $k$ and $n$ be positive integers such that $2k\le n$. If $\mathcal{F}$ is an intersecting family formed by some $k$-element subsets of a given $n$-element set, then
$$|\mathcal{F}|\le \binom{n-1}{k-1}.$$
\end{theorem}

The bound in the above theorem is tight, shown by the family of all $k$-element sets containing a fixed element. In fact, if $n>2k$ all the extremal constructions are of this form.

For a family $\mathcal{H}$ and element $x$, the {\it star} centered at $x$ is defined as $\mathcal{H}_x:=\{H \in \mathcal{H}~:~x\in H\}$. Note that stars are always intersecting. We say that a family $\mathcal{H}$ has the Erd\H{o}s--Ko--Rado property (or that $\mathcal{H}$ is EKR) if there is an element $x$ such that $|\mathcal{F}|\le |\mathcal{H}_x|$ holds for any intersecting subfamily $\mathcal{F}\subset\mathcal{H}$.

The Erd\H{o}s--Ko--Rado theorem can can generalized in several ways (see \cite{gm2015} or Chapter 2 of \cite{gp2018} for an overview). One of these is determining whether certain families are EKR or not. In this note, we will prove that the family defined below is EKR.

Let $M_n$ be a simple graph with a vertex set $V(M_n)=\{a_1, a_2, \dots, a_n, b_1, b_2,\dots, b_n\}$ and edge set $E(M_n)=\{(a_1,b_1), (a_2,b_2), \dots, (a_n,b_n)\}$. Let $\mathcal{H}^{(p,s)}(n)$ be the family of those subsets of $V(M_n)$ that contain exactly $2p+s$ vertices and span exactly $p$ edges. In other words, they contain $p$ disjoint edges and $s$ isolated vertices. It is easy to see that $|\mathcal{H}^{(p,s)}(n)|=\binom{n}{p}\binom{n-p}{s}2^s$.

In a recent manuscript Fuentes and Kamat made a conjecture about $\mathcal{H}^{(p,s)}(n)$ and proved it for a certain range of parameters.

\begin{conjecture}[Fuentes, Kamat (2024) \cite{fk2024}] \label{fkconj}
For non-negative integers $p,s$ satisfying $1\le 2p+s\le n$, $\mathcal{H}^{(p,s)}(n)$ is EKR.
\end{conjecture}

\begin{theorem}[Fuentes, Kamat (2024) \cite{fk2024}] \label{fkthm}
Conjecture \ref{fkconj} is true if $n\ge 2p+2s$, or $n=2p+s$, or $s=1$.
\end{theorem}

They also observed that the case $s=0$ is equivalent to the classic Erd\H{o}s--Ko--Rado theorem, while the case $p=0$ is an EKR-type result for intersecting families of independent sets in $M_n$, proved independently by Deza-Frankl \cite{df1983} and Bollob\'as-Leader \cite{bl1997}. The latter pair of authors applied this result to a diameter problem in the grid. Conjecture \ref{fkconj} does not hold for $n<2p+s$, since in that case $\mathcal{H}^{(p,s)}(n)$ is an intersecting family in its entirety.

The main result of this note is proving Conjecture \ref{fkconj}.

\section{Main result}

\begin{theorem} \label{mainthm}
For non-negative integers $p,s$ satisfying $1\le 2p+s\le n$, $\mathcal{H}^{(p,s)}(n)$ is EKR.
\end{theorem}

\begin{proof}
All sets of $\mathcal{H}^{(p,s)}(n)$ are of size $2p+s$ and all vertices of $M_n$ are included in the same number of sets of $\mathcal{H}^{(p,s)}(n)$. This implies that the number of sets of $\mathcal{H}^{(p,s)}(n)$ containing any given vertex $x$ is $$|\mathcal{H}_x^{(p,s)}(n)|=\frac{2p+s}{2n}\cdot |\mathcal{H}^{(p,s)}(n)|.$$
We will show that the above quantity is an upper bound for the size of any intersecting subfamily of $\mathcal{H}^{(p,s)}(n)$, therefore $\mathcal{H}^{(p,s)}(n)$ is EKR.

Let $\mathbb{Z}_n$ denote the $n$-element cyclic group with elements $\{0,1, \dots, n-1\}$. We define the double cycle $D_n$ as the set $\{(x,y)~:~x\in \mathbb{Z}_n,~y\in\{0,1\}\}$. A function $\varphi: V(M_n)\rightarrow D_n$ is called a {\it proper mapping} if it is a bijection that maps $a_i$ and $b_i$ into two elements of $D_n$ with the same first coordinate for all $i=1,2, \dots n$.

Our proof follows the idea of Fuentes and Kamat \cite{fk2024}. They used a variant of Katona's cycle method, originally developed to give an elegant proof \cite{k1972} of the Erd\H{o}s--Ko--Rado theorem \cite{ekr1961}. The plan is to define subsets of size $2p+s$ called {\it quasi-intervals} in $D_n$, then bound the size of an intersecting family formed by them. Then we will use this result and the notion of proper mappings to give an upper bound on the size of intersecting families in $\mathcal{H}^{(p,s)}(n)$. We consider two cases based on the parity of $s$.

{\bf Case 1:} $s=2k$. We define $n$ subsets of $D_n$ called quasi-intervals as follows. For $i=0,1, \dots, n-1$, let
$$B_i=\{(i-k,0), (i-k+1,0), \dots, (i+p-1,0)\}\cup \{(i,1), (i+1,1), \dots, (i+p+k-1,1)\}$$

Note that all quasi-intervals are of size $2(p+k)=2p+s$. The quasi-intervals intersecting a given $B_i$ are $\{B_{i-p-k+1}, B_{i-p-k+2}, \dots B_{i-1}, B_{i+1}, B_{i+2}, \dots B_{i+p+k-1}\}$, where the indexing of the quasi-intervals is considered modulo $n$. Also observe that the sets $B_{i-p-k+c}$ and $B_{i+c}$ are disjoint for all $c=1,2, \dots, p+k-1$. This implies that the size of an intersecting family formed by quasi-intervals is at most $p+k$.

Let $\mathcal{F}\subset \mathcal{H}^{(p,s)}(n)$ be an intersecting family. Let $S$ denote the number of pairs $(H, \varphi)$ where $H\in\mathcal{F}$, $\varphi$ is a proper mapping and $\varphi(H)$ is a quasi-interval. Let $f(n,p,k)$ denote the number of proper mappings that take a given $H\in\mathcal{H}^{(p,s)}(n)$ to a given quasi-interval $B_i$. By symmetry, $f(n,p,k)$ does not depend on the choice of $H$ and $B_i$. Since any proper mapping takes exactly one set of $\mathcal{H}^{(p,s)}(n)$ to $B_0$, the number of proper mappings is $|\mathcal{H}^{(p,s)}(n)|\cdot f(n,p,s)$.

We can calculate $S$ by first choosing $H\in\mathcal{F}$, then a quasi-interval $B_i$ and finally a proper mapping taking $H$ to $B_i$. This gives us
$$S=|\mathcal{F}|\cdot n \cdot f(n,p,s).$$

On the other hand, we can give an upper bound on $S$ by first choosing $\varphi$ then $B_i$, for which we have at most $p+k$ choices since the quasi-intervals in $\varphi(\mathcal{F})$ form an intersecting family. Therefore
$$S\le |\mathcal{H}^{(p,s)}(n)|\cdot f(n,p,s)\cdot (p+k).$$
After rearranging,

$$|\mathcal{F}|\le \frac{p+k}{n}\cdot|\mathcal{H}^{(p,s)}(n)|=\frac{2p+s}{2n}\cdot|\mathcal{H}^{(p,s)}(n)|.$$

{\bf Case 2:} $s=2k+1$. In this case we will define $2n$ quasi-intervals as follows. For $i=0,1, \dots, n-1$, let
$$B_{2i}=\{(i-k-1,0), (i-k,0), \dots, (i+p-1,0)\}\cup \{(i,1), (i+1,1), \dots, (i+p+k-1,1)\}$$
$$B_{2i+1}=\{(i-k,0), (i-k+1,0), \dots, (i+p-1,0)\}\cup \{(i,1), (i+1,1), \dots, (i+p+k,1)\}$$

Note that all quasi-intervals are of size $(p+k)+(p+k+1)=2p+s$. The quasi-intervals intersecting a given $B_i$ are $\{B_{i-2p-2k}, B_{i-2p-2k+1}, \dots B_{i-1}, B_{i+1}, B_{i+2}, \dots B_{i+2p+2k}\}$, where the indexing of the quasi-intervals is considered modulo $2n$. Also observe that the sets $B_{i-2p-2k-1+c}$ and $B_{i+c}$ are disjoint for all $c=1,2, \dots, 2p+2k$. This implies that the size of an intersecting family formed by quasi-intervals is at most $2p+2k+1=2p+s$.

Using the same double counting argument to finish the proof, it follows that
$$|\mathcal{F}|\le \frac{2p+s}{2n}\cdot|\mathcal{H}^{(p,s)}(n)|.$$

\end{proof}


\end{document}